\documentclass[a4paper,10pt]{amsart}
\usepackage[latin1]{inputenc}  
\usepackage[T1]{fontenc} 
\usepackage{verbatim}
\usepackage{amsmath}
\usepackage{amssymb}
\usepackage{amscd}
\usepackage[arrow,matrix,curve,ps]{xy}
\usepackage{tikz}
\usepackage{hyperref}
\xyoption{all}
\input xy
\newcommand{\C}{\mathbb{C}}
\newcommand{\A}{\mathbb{A}}
\renewcommand{\P}{\mathbb{P}}
\newcommand{\R}{\mathbb{R}}

\newcommand{\G}{\mathbb{G}}
\newcommand{\Q}{\mathbb{Q}}
\newcommand{\N}{\mathbb{N}}
\newcommand{\Z}{\mathbb{Z}}

\def \maps#1 {\ {\buildrel #1  \over \mapsto }\ }

\newtheorem{theorem}{Theorem}[section]
\newtheorem{lemma}[theorem]{Lemma}
\newtheorem{corollary}[theorem]{Corollary}
\newtheorem{proposition}[theorem]{Proposition}
\newtheorem{definition}[theorem]{Definition}

\theoremstyle{definition}
\newtheorem{remark}[theorem]{Remark}
\newtheorem{example}[theorem]{Example}
\newtheorem{question}[theorem]{Question}

\title{Motives of graph hypersurfaces with torus operations}
\author{Stefan M\"uller-Stach} \author{Benjamin Westrich}
\email{mueller-stach@uni-mainz.de} \email{b.westrich@gmx.net}
\address{Institut f\"ur Mathematik, Fachbereich 08, Johannes Gutenberg-Universit\"at Mainz, 55099 Mainz, Deutschland}

\begin{document}

\begin{abstract}
We study conditions under which graph hypersurfaces 
admit algebraic torus operations. This leads in principle to a computation of graph motives 
using the theorem of Bialynicki-Birula, provided one knows the fixed point loci in a resolution of singularities.  
\end{abstract}

\maketitle

\section*{Introduction}

Feynman diagrams and their amplitudes are of fundamental importance in perturbative quantum field theory. Extensive calculations 
of these amplitudes for graphs of low loop numbers by Broadhurst and Kreimer in \cite{BKR1} and \cite {BKR2} revealed the motivic nature 
of these amplitudes, showing that in many cases they are expressible as rational linear combinations of multiple zeta values. This brought up the question whether
all Feynman amplitudes evaluate to multiple zeta values. By general principles \cite{DG,L}, this would mean that 
Feynman amplitudes are periods of mixed Tate motives. Kontsevich \cite{Kont} related this to point counting on the hypersurface 
defined by the singularities of the integrand in the Feynman amplitude. Despite the empirical evidence created by Stembridge in \cite{Stem}, 
Belkale and Brosnan showed that the point counting function for general graph hypersurfaces is not of polynomial type, in fact, 
all point counting functions can be expressed in terms of those of graph hypersurfaces \cite{BelBros}. 
Bloch, Esnault and Kreimer \cite{BEK} investigated the foundations of Feynman amplitudes and their relations to periods of mixed Hodge structures,
and studied the mixed Hodge structure of the middle cohomology for wheel-type graphs. Explicit graphs not of mixed Tate type have first been found by 
Brown-Schnetz \cite{Brown, BroSchn} and Doryn \cite{Dor}. 

The intention of this paper is to explore torus operations on graph hypersurfaces $X_\Gamma$ and their non-singular models, 
and to provide a set-up for studying the resulting motive using the theorem of Bialynicki-Birula \cite{BB}. 
In section \S 1, we give criteria for the existence of algebraic torus operations. In \S 2, we 
focus on a particular class of graphs, obtained by a glueing process, where the torus operation is evident. In \S 3, we use the derived category $DM(k)$ of motives and 
apply the theorem of Bialynicki-Birula in a motivic context in order to study the motive of $X_\Gamma$. The presence of a torus operation 
reduces the complexity of the motive of $X_\Gamma$ with this method to that of the fixed point loci in some resolution of singularities.

\section{Existence of torus operations on graph hypersurfaces}

\begin{definition} Let $\Gamma$ be a finite, connected, not necessarily simple graph. The graph polynomial $P_\Gamma$ of $\Gamma$ is defined as 
$$
P_\Gamma:= \sum_\tau \prod_{e \notin \tau} X_e,
$$
where $\tau$ runs through all spanning trees of $\Gamma$, and $X_e$ is a polynomial variable for each edge $e \in E(\Gamma)$. 
The polynomial $P_\Gamma$ is homogenous of degree $h=h_1(\Gamma)$ \cite{BEK}. We define the graph hypersurface 
$$
X_\Gamma:=\{P_\Gamma=0 \} \subset \P^{n-1}, \quad n = \sharp E(\Gamma).
$$ 
\end{definition}

In \cite{BEK}, this polynomial was rewritten in terms of a determinant of a symmetric $(h \times h)$-matrix $M_\Gamma$ with linear entries. 
Since much of this paper relies on this description we will repeat it here. For $\Gamma$ we choose an orientation of its edges. 
Define a map $\partial \ : \ \Z^{E(\Gamma)} \rightarrow \Z^{V(\Gamma)}$, by $e \mapsto \sum_{v \in V(\Gamma)} {\rm sgn}(v,\, e) v$, 
where ${\rm sgn}(v,\,e) = 1$ if $v$ is the source of the edge $e$, further ${\rm sgn}(v,\, e) = -1$ if $v$ is the target of $E$. 
This gives rise to a simplicial complex $\Z^{E(\Gamma)} \stackrel{\partial}{\rightarrow} \Z^{V(\Gamma)}$ and a corresponding exact sequence 
\[0 \rightarrow H_1(\Gamma,\, \Z) \stackrel{\iota}{\rightarrow} \Z^{E(\Gamma)} \rightarrow \Z^{V(\Gamma)} \rightarrow H_0(\Gamma,\, \Z) \rightarrow 0.\] 
Let $l_e(\cdot),\, e \in E(\Gamma)$, denote the dual basis of the standard basis of all edges 
$e \in E(\Gamma) \subseteq \Z^{E(\Gamma)}$. Then we can consider the bilinear forms $q_e$ of rank $1$ given by 
\[
q_e:=(l_e \circ \iota) \cdot (l_e \circ \iota) \ : \ H_1(\Gamma,\, \Z) \times H_1(\Gamma,\, \Z) \rightarrow \Z.
\]
Choose a basis $B=(c_1,\, \ldots,\, c_{h_1(\Gamma)})$ of $H_1(\Gamma,\,\Z)$, let $M_e = M_e(B)$ be the Gram matrix associated to $q_e$, and set
\[
M_{\Gamma,B} := \sum_{e \in E(\Gamma)} X_eM_e \in \Z[X_e\colon  e \in E(\Gamma)]_1 \otimes_{\Z} {\rm End}(\Z^{h_1(\Gamma)}).  
\]
Here, $\Z[X_e\colon e \in E(\Gamma)]_1$ denotes the degree $1$ part of the algebra $\Z[X_e:  e \in E(\Gamma)]$. 
We will usually abuse the notation and write $M_\Gamma$ without the basis in the subscript.

\begin{lemma}\label{BEKMatrix} One has $P_\Gamma=\pm \det(M_\Gamma)$.
\end{lemma}

\begin{proof} 
See [\cite{BEK},\, Proposition (2.2)].
\end{proof} 

In this description of $M_\Gamma$, the diagonal entries contain sums of variables $X_e$ of edges $e$ contained in cycles 
in the basis $B$ of $H_1(\Gamma,\Z)$. The entries $M_{ij}$ of $M_\Gamma$ with $i<j$ contain variables $X_e$ of edges which 
form the glueing data for the basis elements $c_i, c_j \in B$. 

\begin{lemma} \label{MatrixStructure}
The diagonal entries of $M_{\Gamma}$ generate a free $\Z$-submodule of rank $h_1(\Gamma)$ in $\Z[X_e\colon e \in E(\Gamma)]_1$ with free complement. 
In particular, if we tensor the same $\Z$-module with an arbitrary field $k$, then the diagonal entries remain $k$-linear independent.
\end{lemma}

\begin{proof} There is an isomorphism of free modules $\Z^{E(\Gamma)} \longrightarrow \Z[X_e\colon e \in E(\Gamma)]_1$ under the assignment $e \mapsto X_e$. 
Also, there is a basis $B = (c_1,\, \ldots,\, c_{h_1(\Gamma)})$ of the free submodule $H_1(\Gamma,\, \Z) \subset \Z^{E(\Gamma)}$, where each $c_i$ is a cycle in $\Gamma$.
The entries of $M_{\Gamma,B}$ on the diagonal are precisely the sums $\sum_{e \in c_i} X_e$. They generate a free submodule $H$ of 
$\Z[X_e\colon e \in E(\Gamma)]_1$, and we obtain an isomorphism of free submodules 
\[
H_1(\Gamma,\, \Z) \longrightarrow H.
\]
Any spanning tree $\tau$ in $\Gamma$ induces a free submodule $T$ in $\Z^{E(\Gamma)}$ generated by the edges of $\tau$. Since 
the union of $\tau$ with any edge outside of $\tau$ contains a cycle, this free submodule is complementary to $H_1(\Gamma,\, \Z)$. This shows that 
$H$ has a free complement inside $\Z[X_e\colon e \in E(\Gamma)]_1$. 
\end{proof}

\begin{remark} \label{SpanningTrees}
Subdivision of edges gives rise to affine fiber bundles over $X_{\Gamma}$: Let $\Gamma'$ be the graph obtained 
from $\Gamma$ by subdividing the edge $e$ into $e_1$ and $e_2$. Then $P_{\Gamma'} = P_\Gamma|_{X_e = X_{e_1}+X_{e_2}}$. 
\end{remark}

From now on, we choose a field $k$ of characteristic $\neq 2$, and consider $M_\Gamma$ as a matrix in $k[X_e\colon e \in E(\Gamma)]_1$ after extending scalars.
Then the diagonal entries remain linearly independent over $k$. 

We would like to study torus operations on the graph hypersurface 
\[
X_\Gamma=\{P_\Gamma=0\} \subset \P^ {n-1}.
\]
Let us start with some preliminary remarks on torus operations in general.

\begin{definition}[Torus operation]
A linear operation of the $r$-dimensional torus $\G_m^r$ on $\P^{n-1}$ is given by a homomorphism
of group varieties (over $k$) 
\[
\G_m^r \longrightarrow  {\rm PGL}_n.
\]
We define a torus operation of an $r$-dimensional torus $\mathbb{G}_m^r$ to be of rank $s$ on a closed subvariety $X \subset \P^{n-1}$, 
if the image under the associated morphism
\[
\G_m^r \longrightarrow {\rm Aut}(X) \cap {\rm PGL}_n
\]
is of dimension $s$. An operation is called faithful, if it is injective. 
In particular, a faithful operation of an $r$-dimensional torus on a variety $X$ is of rank $r$.
We denote by ${\rm Fix}_{X}(\G_m^r)$ the fixed point locus of the operation on $X$. 
\end{definition}

\begin{remark}
(a) In our situation, all tori are split over $k$, i.e., isomorphic to $\G_m^r$, which is why we defined them in this restricted way. \\
(b) Any linear operation of a torus is diagonizable, i.e., the image is conjugate (w.r.t. linear transformations) to a subgroup of the 
maximal standard torus in ${\rm PGL}_n(k)$.\\
(c) Consider an operation $\G_m =: T \rightarrow {\rm Aut}(X)$, where $X$ is projective. 
Assume that it is the restriction of a linear, diagonal operation on some embedding $X \hookrightarrow \P^{n-1}$. 
Then the operation $T$ is determined by a vector $\eta=(\eta_1,...,\eta_n)$ of integers - usually called the weights of the operation - such that 
the operation is given by $X_i \mapsto t^{\eta_i} X_i$ for $1 \le i \le n$. 
\end{remark}

From now on, we assume that $X=\{f=0\} \subset \P^{n-1}$ is a hypersurface. We assign to each such $f$ a $\Z$-module as follows:

\begin{definition}[Weight lattice]
\label{lattices}
Let $s \in \N_0$ and $k[X_1,\ldots,\, X_n]_s$ be the vector space of homogeneous polynomials in $n$ variables of degree $s$. 
For any 
\[
f = \sum_{|\alpha| = s} c_{\alpha}X^{\alpha} \in k[X_1,\, \ldots,\, X_n]_s
\]
we call 
\[
\Lambda_f:= \left\{ \eta \in \Z^n \ : \ \eta \cdot \alpha = c  \text{ independent of } \alpha  \text{ for any } \alpha \text{ with }  c_{\alpha} \neq 0\right\}
\]
the weight lattice of the polynomial $f$. We will call the elements of $\Lambda_f$ 
the weight vectors of $f$. Each weight vector $\eta \in \Lambda_f$ gives rise to a $\G_m$-operation on $X=\{f=0\} \subset \P^{n-1}$ given 
by $X_i \mapsto t^{\eta_i} X_i$. An operation corresponding to a weight vector $\omega$ if trivial on $\P^{n-1}$ if and only 
if $\omega \in \Z \cdot (1,\ldots,1)$. 
Since we consider homogeneous polynomials $f$, we have $\Z \cdot (1,\ldots,1) \subseteq \Lambda_f$.
Thus, the weight lattice $\Lambda_f$ has rank $r+1$ if and only if $X$ admits a linear rank $r$ torus operation in the coordinates $X_i$.
\end{definition}

\begin{proposition} \label{mainexample}
Let $M$ be the generic symmetric $(h \times h)$-matrix over the polynomial ring $k[Y_{ij} \mid 1 \leq i \leq j \leq h]$. 
Its entries are $M_{ij}=Y_{ij}$ for $i \le j$ and $M_{ij}=Y_{ji}$ for $i>j$. Then the number of variables is $n=\binom{h+1}{2}$, and 
$f=\det(M)$ has the weight lattice 
\[
\Lambda_h = \{ \omega = (\omega_{ij})_{i \leq j} \in \Z^n \ : \ 2\omega_{ij} = \omega_{ii}+\omega_{jj} \text{ for all } i<j \}.
\]
The lattice $\Lambda_h$ has rank $h$, since an element is already determined by the diagonal entries. The corresponding torus operation on $X=\{\det(M)=0\}$ 
is given by 
\[
Y_{ij} \mapsto t_it_jY_{ij} 
\]
for an element $(t_1,\ldots,t_h)$ in the diagonal torus $\G_m^{h-1} \subset {\rm PGL}_h$ of rank $h-1$. 
\end{proposition}

\begin{proof} The explicitly defined torus operation above shows immediately that $\Lambda_h \subseteq \Lambda_{\det(M)}$. To prove the converse, one looks at 
\[
\det(M)=\sum_{\sigma \in \Sigma_h} {\rm sgn}(\sigma) Y_{1\, \sigma(1)} \cdots Y_{h\, \sigma(h)}.
\]
The permutation $\sigma={\rm id}$ corresponds the monomial $Y_{1\, 1} \cdots Y_{h\, h}$ in $\det(M)$, the transposition $(ij)$
corresponds to $Y_{1\, 1} \cdots Y_{i\, j} \cdots Y_{j\, i} \cdots Y_{h\, h}$. Therefore, the definition of the weight lattice $\Lambda_{\det(M)}$ implies that 
$\Lambda_{\det(M)} \subseteq \Lambda_h$. 
\end{proof}

Let us now look at the special case where the polynomial $f=\det(M)$ is the determinant of a symmetric $(h \times h)$-matrix $M$ with 
linear homogenous polynomials in $k[X_1,\ldots,\, X_n]$ as entries. A useful invariant is: 

\begin{definition}
Let $\ell(M)$ be the $k$-dimension of the span of all non-zero entries $M_{ij}$ in the upper triangle of $M$. (By the upper triangle we mean 
the entries $M_{ij}$ of $M$ with $i \leq j$.)
\end{definition}

We now construct torus operations on such determinant hypersurfaces. We first look at the special case where all 
entries are linearly independent linear homogenous polynomials and the matrix is full.

\begin{proposition} \label{GenericOperation}
Suppose $M=(M_{ij}) \in k[X_1,\ldots,\, X_n]^{h\times h}$ is a full symmetric matrix, such that the entries $M_{ij}$ in the upper 
triangle are all non-zero and linearly independent linear homogenous polynomials. Assume also that $n=\ell(M)=\binom{h+1}{2}$ 
is the number of all entries. Consider the maximal torus $\mathbb{G}_m^{h-1}(k) \subset {\rm PGL}_h(k)$
given by diagonal matrices. Then the variety 
\[
X := \{{\rm det}(M) = 0\} \subseteq \P^{n-1}
\]
carries a $\mathbb{G}_m^{h-1}$-operation of rank $h-1$. 
\end{proposition}

\begin{proof} By assumption, there is an isomorphism
\[
k[Y_{ij} \mid i \leq j] {\buildrel \cong \over \longrightarrow}  k[X_1,\ldots,\, X_n]
\]
by substitution, and hence a linear $k$-isomorphism
\[
{\rm Proj}\, k[X_1,\ldots,\, X_n]  {\buildrel \cong \over \longrightarrow} {\rm Proj}\, k[Y_{ij} \mid i \leq j]
\]
between projective spaces of dimension $n-1$.
Thus we may work with the $Y_{ij}$-coordinates and may assume that we are in the situation of Example~\ref{mainexample}. 
Define a torus operation on ${\rm Proj}\, k[Y_{ij} \mid i \leq j]$
by $Y_{ij} \mapsto t_it_j Y_{ij}$.  We have to show that 
this operation is of rank $h-1$. But the point $P=(1:\ldots:1)$ lies 
on $X$ and its orbit is given by the image of the morphism
\[
\varphi: (t_1,\ldots,t_h) \mapsto (t_1^2:\ldots:t_h^2:t_1t_2: \ldots) \in X. 
\]
Taking the differential of $\varphi$ at $P$, we see that it is an immersion. Hence $\dim {\rm Im}(\varphi)=h-1$.   
\end{proof}

In general, one has $n \ge \ell(M)$ and we get a slightly more general result for matrices which have some vanishing entries off the diagonal 
but all non-zero entries are linearly independent linear homogenous as above.

\begin{theorem}\label{Existence}
Let $M =(M_{ij}) \in k[X_1,\ldots,\, X_n]^{h\times h}$ be a symmetric matrix such that all non-zero entries $M_{ij}$ for $i \le j$ 
are linearly independent linear homogenous polynomials, and all diagonal entries $M_{ii}$ are non-zero. Then the hypersurface 
\[
X:=\{\det(M)=0 \} \subset \P^{n-1}
\]
admits a linear $\mathbb{G}_m^r$-operation with $r=h-1+ n -\ell(M)$.

This operation is of rank $r$, if there is a point $P \in X$ such that $P$ is not contained in the union of the linear hypersurfaces defined 
by the diagonal entries of $M$. 
\end{theorem}

\begin{remark} For the torus operation defined in Example~\ref{GenericOperation}, the number $r=h-1+ n -\ell(M)$ is 
maximal with this property. However, there may be examples with extra operations, see example~\ref{extra_example}.
\end{remark}

\begin{proof} Let us first assume that $n=\ell(M)$. 
As in the proof of Prop.~\ref{GenericOperation}, we may work with the variables $Y_{ij}$, and assume that $X=\det(Y_{ij})$,
where some variables $Y_{ij}$ for $i \neq j$ are set to be zero. The $\G_m^{h-1}$-operation from Prop.~\ref{GenericOperation}
can be restricted to $X$, since $X$ is the zero locus of the $\G_m^{h-1}$-invariant hyperplanes $Y_{ij}=0$. 
Therefore, the determinantal hypersurface $X$ admits an operation of $\G_m^{h-1}$ defined by the weight lattice $\Lambda_h$ from Prop.~\ref{GenericOperation}.
To show that the operation is still of rank $h-1$ in this case, where some entries vanish, look at the given point $P=(P_{i,j}) \in X$. 
Let $\Sigma$ be the set of all indices $i \le j$ such that the entry $M_{ij}$ in $M$ is non-zero.  
Consider the morphism
\[
\varphi_\Sigma:  \mathbb{G}_m^{h-1} \longrightarrow X \subset \P^{n-1}, \quad (t_1,...,t_h) \mapsto (P_{i,j}t_it_j)_{(i,j) \in \Sigma}  \in X.
\]
Differentiating at $P$ as in Prop.~\ref{GenericOperation}, we see that the Jacobi matrix of $\varphi_\Sigma$ contains
a diagonal submatrix of rank $h$, since all $P_{ii}$ are non-zero by assumption. 

Suppose now that $n>\ell(M)$. The matrix $M$ defines a $k$-linear surjection 
\[
q: \tilde X \longrightarrow \tilde X_\Sigma
\]
of the affine cone $\tilde X$ over $X$ to $\tilde X_\Sigma$, the affine cone of $X_\Sigma \subset \P^{\ell(M)-1}$ which is the 
determinantal hypersurface defined by the symmetric matrix with non-zero entries $Y_{i,j}$ for $(i,j) \notin \Sigma$. 
Since $q$ is induced by the projection 
\[
\A^n \longrightarrow \A^{\ell(M)},
\]
the morphism $q$ is a trivial vector bundle of rank $n-\ell(M)$ whose fibers are linearly embedded 
in $\A^{n}$. We have already shown that $\tilde X_\Sigma$ admits a rank $h-1$ torus operation. 
The $\G_m^{h-1}$-operation on $X_\Sigma$ induces a $\G_m^{h-1} \times \G_m^{n-\ell(M)}$ 
operation on $\tilde X=\tilde X_\Sigma \times \A^{n-\ell(M)}$ which is of rank $r$ when restricted to $X$.
\end{proof}
 
\begin{example} Wheels $WS_h$ with $h$ spokes and $2h$ edges satisfy Theorem~\ref{Existence}, since 
\[
M(WS_h)= \left(\begin{matrix} Y_{11} &  -X_2 & 0 & \cdots & -X_1\cr 
-X_2 & Y_{22} & -X_3 & \cdots & 0 \cr 
0 & -X_3 & \ddots  & \ddots & 0 \cr 
\vdots & & \ddots & \ddots & \vdots \cr 
0 & \dots && Y_{h-1,h-1} &  -X_h \cr 
-X_1 & 0 & \cdots & -X_h & Y_{hh} 
\end{matrix} \right),
\] 
with $Y_{ii}=X_i+X_{i+1}+X_{h+i}$. Here, $i+1$ is to be considered$\mod h$. 

As a consequence, the associated hypersurfaces $X_h$ admits a torus 
operation of rank $h-1$. This bound is sharp, e.g. in the case $h=3$, the hypersurface
$X_3 \subseteq \P^5 \simeq \P(\Gamma(\P^2,\, \mathcal{O}_{\P^2}(2)))$ is the complement of the 
$5$-dimensional homogenous space $PSL_3(\C)/SO_3(\C)$, which admits a 
rank $2$ torus operation. There is no larger linear torus operation, since the group $PSL_3(\C)$ is the stabilizer of 
$\P^5 \backslash X_3$ in ${\rm PSL}_6(\C) = {\rm Aut}(\P^5)$ and it has rank $2$ (see \cite{B}).
\end{example}

In general, the condition of linear independence in Theorem~\ref{Existence} is too restrictive. 
We need to define a new invariant for symmetric matrices $M$ to formulate a more general result.
The proof of Theorem~\ref{Existence} then implies much more as we will see now.

Let us first fix a certain normal form that we need to formulate the setting and the result in an economical manner. 
One can always pass from $M$ to a certain normal form by a unique linear transformation as follows. 
Let $M \in k[X_1,\ldots,\, X_n]_1^{h\times h}$ be a symmetric matrix of linear forms 
such that all diagonal entries are non-zero and linearly independent.
As above, we denote by $\ell(M)$ the dimension of the span of all upper-triangular entries. 
The $h$ diagonal entries of $M$ are linearly independent by assumption, so we label them (in this order) by $X_1,...,X_h$. 
Then we pass to the next parallel diagonal with $i=j-1$. If the entry $M_{12}$ is linearly independent of $X_1,...,X_h$, then we replace it by $X_{h+1}$, otherwise
it is a linear form $L_{12}(X_1,...,X_h)$. We continue in the obvious way by going from top to bottom in all diagonals in the upper triangle. For the entries below 
the diagonal we take the mirror image. Each non-zero entry $M_{ij}$ of $M$ is either a variable 
$X_1,...,X_{\ell(M)}$, if it occurs for the first time, or a linear form $L_{ij}(X_1,...,X_{\ell(M)})$ in those variables. If $L_{ij}$ equals a repeated variable
(which may happen), we nevertheless call it a linear form. Hence, the entries which are called ''variables'' are the first occurences in the chosen ordering. 
We say that the resulting symmetric matrix is in {\it quasi-lexicographic normal form}. 
Note that passing from $M$ to its normal form is a unique algorithm. 

\begin{definition} Let $M$ be in quasi-lexicographic normal form.
We define an equivalence relation on indices $(ij)$ ($i \le j$) of the non-zero entries $M_{ij}$ as the transitive hull of the 
symmetric relation given by 
\[
(ij) \sim (kl) \Leftrightarrow \text{ a common variable } X \in \{X_1,...,X_{\ell(M)} \} \text{ occurs in } M_{ij} \text{ and } M_{kl}.
\] 
The equivalence classes are called clusters.  \\
An element $(ij)$ with $i<j$ in a cluster $C$ is called excessive, if $X_i$ or $X_j$ do not occur in $L_{ij}(X_1,...,X_{\ell(M)})$.
Let 
\[
\delta(M):=\sum_{\text{clusters }C} (|C|-1) + \sharp \text{ excessive entries in } M
\]
be the excess of $M$.
\end{definition}

\begin{theorem}\label{Existence2} 
Let $M$ be in quasi-lexicographic normal form. Then the hypersurface
$$
X:=\{\det(M)=0 \} \subset \P^{n-1},
$$
admits a rank $r$ torus operation which is diagonal in the variables $X_1,...,X_n$, and where
\[
r \geq \max\left(0,h-1 +n- \ell(M) - \delta(M) \right),
\]
if there is a point $P \in X$ such that $P$ is not contained in the union of the linear hypersurfaces defined 
by the diagonal entries of $M$.
\end{theorem}

\begin{proof} By our convention, all variables $X_1,\ldots,X_{\ell(M)}$ occur for the first time at a unique position $M_{ij}$ in $M$,  
and $X_1,\ldots,X_h$ are the diagonal entries. 
Substituting new variables $Y_{ij}$ for each remaining linear form $L_{ij}(X_1,...,X_{\ell(M)})$, we arrive at an inclusion
\[
i: \P^{n-1} \hookrightarrow \P^{N+n-\ell(M)-1},
\]
where $N-\ell(M)$ is the number of additional variables $Y_{ij}$ with $i<j$. This inclusion maps $X$ to a codimension $N-\ell(M)+1$ subvariety  
\[
X'=i(X) = \{\det(M')=0\} \bigcap \{H_{ij}=0\},
\] 
where $M'$ is the matrix obtained by the same substitutions, and $H_{ij}$ are the linear hyperplanes 
\[
H_{ij}= Y_{ij} - L_{ij}(X_1,...,X_{\ell(M)}).
\]
Theorem~\ref{Existence} implies the existence of a torus $T$ of rank $\ge h-1+ N+n -\ell(M) - \ell(M') =h-1 + n-\ell(M)$ acting on $\{\det(M')=0\}$. 
Now we count conditions to estimate the minimal dimension dimension of a torus stabilizing $X'=i(X)$. For the variables $X_i$ in each cluster $C$
to have equal weight amounts to at most $|C|-1$ conditions. The weights $\omega_{ij}$ of the new variables $Y_{ij}$ with $i>j$ are related to the weights of the
diagonal entries by the formula $2\omega_{ij}=\omega_{ii}+\omega_{jj}$. Hence, if $(ij)$ is not excessive, one has $\omega_{ij}=\omega_{ii}=\omega_{jj}$ which 
satisfies the formula. If $(ij)$ is excessive, then the equation $2\omega_{ij}=\omega_{ii}+\omega_{jj}$ 
imposes one new extra condition on the weights $\omega_{ii}$ and $\omega_{jj}$. 

In total, this gives $\delta(M)$ conditions, and hence we obtain a torus operation of rank 
$\ge n-\ell(M)+h-1 -\delta(M)$. 
\end{proof}

\begin{remark} 
We cannot prove that the coordinate system suggested in our proof does always yield a torus operation of the highest possible
rank. For example, there could be an operation which is not diagonal in our chosen coordinates, or the cluster conditions are not independent. 
The latter would be detected in the computations of the weights following the algorithm implicit in the proof though. 
Therefore, the bounds in this theorem are not sharp. We provide a corresponding example below. 
\end{remark}

\begin{example} \label{extra_example} Consider the graph
which is the wheel with $3$ spokes with one additional triangle subdivided (see figure below). This gives rise to the matrix 
\[
M = \left(\begin{matrix} 
X_2+X_6+X_8 & X_2+X_6 & -X_2 & X_2\cr 
X_2+X_6 & X_1+X_2+X_4+X_6+X_7 & -X_1-X_2-X_4 & X_1+X_2\cr 
-X_2 & -X_1-X_2-X_4 & X_1+X_2+X_4+X_5 & -X_1-X_2 \cr 
X_2 & X_1+X_2 & -X_1-X_2 & X_1+X_2+X_3 
\end{matrix}\right).
\]

Substituting as in Theorem \ref{Existence2} we arrive at 
\[
M= \left(\begin{matrix} Y_1 & Y_5 & Y_8 & -Y_8 \cr 
Y_5 & Y_2 & Y_6 & -Y_7 \cr
Y_8 & Y_6 & Y_3 & Y_7 \cr 
-Y_8 & -Y_7 & Y_7 & Y_4 
\end{matrix}\right).
\] 
Obviously we have two clusters of length $2$ and $6$ clusters of length $1$. 
By the theorem this means we can expect $X_\Gamma=\{{\rm det}(M)=0\} \subseteq \P^7$ to have no  torus operation. However, there is a $1$-dimensional operation 
given by the weight vector $\omega := (3,\, -1,\, -1,\, -1,\, 1,\, -1,\, -1,\, 1)$. The algorithm would give the same result, as $Y_7$ and $Y_8$ are in 
excessive positions but impose no extra relation. 
\end{example} 

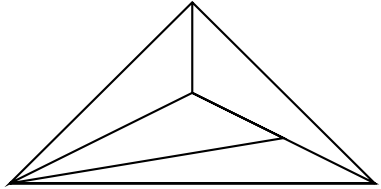
\tikzstyle{every node}=[circle, draw, fill=black!50,inner sep=0pt, minimum width=4pt]
\[
\begin{tikzpicture}[thick,scale=0.8]
	\draw (-3,0)--(3,0)--(0,3)--(-3,0)--(0,1.5)--(1.5,0.75)--(-3,0)--(3,0)--(1.5,0.75)--(0,1.5)--(0,3);
\end{tikzpicture}
\]

\begin{lemma}
\label{Fixpoints}
Let $\Gamma$ be a graph such that the non-zero entries in the upper triangle of $M_\Gamma$ are linearly independent. 
Then, for any faithful operation of $T:=\mathbb{G}_m^{r}$ with $r=h-1+n-\ell(M)$ on $X_{\Gamma}$, as described in Theorem~\ref{Existence}, 
the variety ${\rm Fix}_{\P^{n-1}}(T)$ is zero-dimensional, and consists of points contained in $X_{\Gamma}$.
\end{lemma}

\begin{proof} We may assume that $n=\ell(M)$, since the operation on the $n-\ell(M)$ extra variables is effective.  
By Example~\ref{GenericOperation}, the operation on the generic symmetric matrix with independent linear entries is given by $(t,\, x) \mapsto (t_it_jx_{ij})$. 
Choosing special values for $t_i$ and $t_j$ with $\prod_i t_i=1$, one sees that 
the fixed points in this case are just the points corresponding to the usual standard basis of the underlying space $\P^{N-1}$ with $N=\binom{h+1}{2}$. 
In our more general situation, the graph hypersurfaces are intersections of the generic zero set of the determinant of the generic symmetric matrix with ($T$-invariant) 
linear coordinate subspaces. Hence the fixed point set 
${\rm Fix}_{\P^{n-1}}(T)$ is given by points in $\P^{n-1}$ with exactly one non-zero entry, i.e., a vertex of the coordinate simplex. 
Obviously these points are contained in $X_{\Gamma}$.
\end{proof}

Note that all graph hypersurfaces of wheels $WS_h$ with $h$ spokes satisfy this Lemma.

\section{Examples: $\ast$-graphs}

At the beginning of this section we need to introduce a few conventions. We will call a basis $B \subseteq H_1(\Gamma)$ a cycle basis if it consists only 
of simple cycles. That such a basis exists is a standard fact in graph theory. Since the matrix $M_\Gamma$ associated to a graph $\Gamma$ 
depends on the chosen basis of $H_1(\Gamma)$ we will make this dependence explicit in this section by writing $M_{\Gamma,B}$.\\

A class of examples which have linearly independent entries in $M_{\Gamma,B}$ and which contains the wheels with $n$ spokes are the $*$-graphs: 

\begin{definition}
\label{Stargraph} A polygonal graph $\Gamma$ is a connected, not necessarily simple, graph which has a decomposition 
$\Gamma= \Delta_1 \cup \Delta_2 \cup \cdots \cup \Delta_h$ as a successive glueing (in the sense of topological spaces) along non-empty, connected sets of edges inside 
given cycles $\Delta_i$, and such that no edge is used twice for glueing. Let $E_0 \subset \Gamma$ be the union of all edges used for the glueing. 
A $*$-graph $\Gamma$ is a polygonal graph such that every such decomposition has the property $h_1(E_0) = 0$.
\end{definition}

Note that there are also other, but different, notions of polygonal graphs in the literature.

\begin{example} 
In the literature dealing with the motives of graph hypersurfaces one calls a connected graph $\Gamma$ a {\it banana graph} (denoted by $B_n$) 
if and only if it consists of exactly two vertices and $n$ edges connecting both vertices. This implies that $h_1(\Gamma) = n-1$.
The example of a banana graph with $n=4$ edges and $3$ loops shows that the condition $h_1(E_0) = 0$
depends on the glueing order. To see this, label the edges $1,\, 2,\, 3,\, 4$. This gives as candidates for cycles the graphs consisting of exactly two edges, 
e.g. $(1,\,2)$. Then $B_4 = ((1,\,2) \amalg_{\{2\}} (2,\, 3)) \amalg_{\{3\}} (3,\, 4)$. But also $B_4 = ((1,\,2) \amalg_{\{2\}} (2,\, 3)) \amalg_{\{2\}} (2,\, 4)$. 
Hence, $E_0 = \{2,\, 3\}$ (and $h_1(E_0) = 1$) in the first case and $E_0 = \{2\}$ (and $h_1(E_0) = 0$) in the second. 
This shows that we have to require that $h_1(E_0) = 0$ for all decompositions.
The matrix $M_{\Gamma,B}$ (corresponding to the basis $B$ obtained from the $3$ obvious loops) 
has linearly dependent entries for this graph. One can verify that the hypersurface corresponding to $B_4$ does not admit any non-trivial 
linear $\mathbb{G}_m$-operation. 
\end{example}

\begin{lemma}\label{MVGraphLemma} Assume that $\Gamma$ is a polygonal graph. \\
(i) If there is a decomposition with $h_1(E_0)=0$, then    
\[
h_1(\Gamma)=  \sharp \text{ cycles } \Delta_i=h.
\]
(ii) For all edges $e$ in $\Gamma$, one has 
\[
h_1(\Gamma \setminus e) < h_1(\Gamma). 
\]
\end{lemma}

We will call a graph satisfying (ii) a {\it homology model}. In the literature this is sometimes called $1$-particle irreducible 
without external edges \cite{BEK}. 
We prefer to call it a homology model, since this captures in a better way the topological nature of the definition.

\begin{proof} (i) We use the Mayer-Vietoris Theorem and induction on the number of cycles. Assume 
$\Gamma=\Gamma' \cup \Delta$, where $\Delta$ is a cycle. Then the intersection $\Gamma' \cap \Delta$ 
is a connected and contractible union of edges, in particular $h_1(\Gamma' \cap \Delta)=0$ and $h_0(\Gamma' \cap \Delta)=1$. Hence there is an isomorphism 
$H_1(\Gamma') \oplus H_1(\Delta) \cong H_1(\Gamma)$. \\
(ii) A trivial induction on the decomposition of a polygonal graph shows that $\Gamma \setminus e$ is still connected. 
Let $U$ be an open subset of $\Gamma$ which contains $\Gamma \setminus e$ and is homotopy equivalent to it. 
Also, let $V$ be a contractible open subset containing $e$. Then $U \cup V=\Gamma$, and the assertion follows from the Mayer-Vietoris
sequence for open coverings.   
\end{proof}

While it is natural to define $*$-graphs as polygonal graphs with an additional property, we remark that they form a subclass of planar graphs:

\begin{lemma} \label{poly_versus_planar}
A graph $\Gamma$ is polygonal if and only if it is planar and a homology model.
\end{lemma}

\begin{proof}
$\Gamma$ is polygonal if and only if $\Gamma = \coprod_{E_0} \Delta_i$, where all $\Delta_i$ are simple cycles, every edge belongs to at most two $\Delta_i$'s, 
and no edge in $E_0$ is used twice for glueing. 
This condition means that the set $\{\Delta_i\}$ contains a simple basis of the cycle space $H_1(\Gamma)$ of $\Gamma$. For the definition of a simple basis, 
see \cite[sect. 4.5]{Diestel}. 
Hence $\Gamma$ is planar by MacLane's planarity criterion \cite[Thm. 4.5.1]{Diestel}: a graph is planar if and only if $H_1(\Gamma)$ contains a simple basis. 

Conversely, consider a plane embedding $\Gamma \rightarrow \R^2$. Choose a compact disc $\Gamma \subseteq D \subseteq \R^2$ such that $\partial D \cap \Gamma = \emptyset$ 
(here $\partial$ means ''boundary of''). Define the equivalence relation $\sim$ on $D \times D$ by requiring $x \sim y$ if and only if $x$ and $y$ are connected by a 
path inside $D \backslash \Gamma$ or inside $\Gamma$. This gives a partition $D = \Gamma \cup A \cup \bigcup_{i=1}^{h_1(\Gamma)} \Delta_i^\circ$, where $A$ is the unique 
component with $\partial D \subseteq A$ and $\partial \Delta_i$ are cycles \cite[Prop. 4.2.6]{Diestel}. Then, 
$(\partial \Delta_1^\circ,\, \ldots,\, \partial\Delta_{h_1(\Gamma)}^\circ)$ is a cycle basis 
of $H_1(\Gamma)$, and no edge of $\Gamma$ lies in more than two $\partial \Delta_i$ \cite[Lem. 4.2.2]{Diestel}. 
Since $\Gamma$ is a homology model, every edge of $\Gamma$ is contained in some $\partial \Delta_i$. 
Hence, glueing in the given order shows that $\Gamma$ is polygonal. 
\end{proof}

\begin{definition}
We will call a simple cycle $\Delta \subseteq \Gamma$ an inner cycle of $\Gamma$ if there exist simple cycles $\Delta_2,\, \ldots,\, \Delta_{h_1(\Gamma)}$ such that 
$B := (\Delta,\, \Delta_2,\, \ldots,\, \Delta_{h_1(\Gamma)})$ is a cycle basis of $H_1(\Gamma)$ and 
$\Delta = \sum_{i=2}^{h_1(\Gamma)} \Delta \cap \Delta_i \in H_1(\Gamma, \mathbb{F}_2)$.
\end{definition}

\begin{lemma}\label{inner_cycles} A graph $\Gamma$ with no inner cycles is planar.%
\end{lemma}

\begin{proof}
Note that the class of graphs without inner cycles is closed under taking subgraphs and that (all subdivisions of) the complete bipartite graph $K_{3,\,3}$ and the complete
graph $K_5$ have inner cycles. 
Thus the assertion follows from Kuratowski's planarity criterion that states that a graph is planar if and only if it does not contain neither 
$K_{3,\,3}$ (complete bipartite graph) nor $K_5$ (complete graph) \cite[Thm. 4.4.6]{Diestel}.
\end{proof}

The converse does not hold, since a typical graph with an inner cycle is 

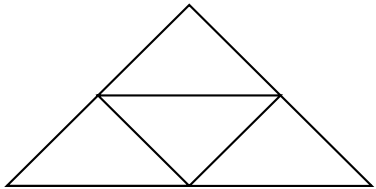
\tikzstyle{every node}=[circle, draw, fill=black!50,inner sep=0pt, minimum width=4pt]
\[
\begin{tikzpicture}[thick,scale=0.8]
	\draw (-3,0)--(3,0)--(0,3)--(-3,0)--(0,0)--(1.5,1.5)--(-1.5,1.5)--(0,0);
\end{tikzpicture}
\]
This graph is not a $\ast$-graph, as $E_0$ is the inner triangle. For $\ast$-graphs, the following characterization holds.

\begin{theorem}\label{polys}  
Let $\Gamma$ be a graph. Then the following conditions are equivalent:\\
(i) $\Gamma$ is a $\ast$-graph. \\
(ii) $\Gamma$ is a homology model, and there exists a cycle basis $B \subseteq H_1(\Gamma)$ such that the non-zero upper-triangular matrix entries 
$M_{ij}$ of $M_{\Gamma,B}$ are linearly independent polynomials in $k[X_1,\, \ldots,\, X_n]_1$.
\end{theorem}

\begin{proof} 
(ii) $\Rightarrow$ (i): We will first show that $\Gamma$ is planar. To this end, we show that
$\Gamma$ has no inner cycles, hence $\Gamma$ is planar by Lemma~\ref{inner_cycles}.
Suppose $\Gamma$ has inner cycles. This means that, in addition to the cycle basis $B = (\Delta_1,\, \ldots,\, \Delta_{h_1(\Gamma)})$, there is 
another cycle basis 
$B' = (\Delta_1',\, \ldots,\, \Delta_{h_1(\Gamma)}')$ of $H_1(\Gamma,\, \mathbb{F}_2)$ such that 
\[
\Delta_1' = \sum_{i=2}^{h_1(\Gamma)} \Delta_1' \cap \Delta_i'.
\] 
In the special case where $B=B'$, this relation immediately leads to a linear dependence between the matrix entry $M_{11}$ and 
other entries in the first row or column, and hence contradicts the assumption. 

In general, since $GL(H_1(\Gamma,\, \mathbb{F}_2))$ is generated by transvections, we can always 
find $t \in GL(H_1(\Gamma,\, \mathbb{F}_2))$ such that $t(B) = B'$, and $t$ is product $t=t_1 \cdots t_l$ of transvections. 
In addition, we will now show that we can reduce to the case where $t_i\cdots t_l(B)$ is a cycle basis for all $i$. 
In the following, we shall do only one iteration of the reduction, since one obtains the full reduction by simply repeating this step. Hence, assume that 
$\Delta_i = \Delta_i'$ for all $i > 1$ and $\Delta_1' = \sum_{i=1}^{h_1(\Gamma)} \alpha_i \Delta_i$, with $\alpha_1 = 1$. For $i= 2, \ldots,\, h_1(\Gamma)$ 
define $t_i = 1 + \alpha_i E_{1i}$, where $E_{1i}$ is the matrix with $1$ at entry $(1,\,i)$ and $0$ else. Then $t(B) = B'$, where $t = \prod_{i=2}^{h_1(\Gamma)} t_i$. 
Note that the $t_i$ commute pairwise. Suppose (after reordering if necessary) for some $i > 2$ (if $i=2$ we are done) that $t_j \cdots t_{h_1(\Gamma)}(B)$ is a cycle 
basis for all $i \leq j$. Then there exists $2 \leq k < i$ such that $t_i \cdots t_{h_1(\Gamma)}(\Delta_1)$ shares edges with the cycle $\Delta_{k}$, since otherwise 
$\Delta_1'$ would not be a simple cycle. Now, swap the indices of $t_{i-1}$ and $t_k$ and proceed inductively. 

Having shown this reduction for $t(B)=B'$, and assuming $B \neq B'$, this reduces us without loss to the situation 
$\Delta_1' = \Delta_1 + \Delta_2$, $\Delta_j'=\Delta_j$ for $j \ge 2$, and $\Delta_1 \cap \Delta_2 \neq \emptyset$. 
Hence, $\Delta_1' \cap \Delta_j = (\Delta_1 \cap \Delta_j)+(\Delta_2 \cap \Delta_j)$ 
for all $j$. In particular, $\Delta_1' \cap \Delta_2 = (\Delta_1 \cap \Delta_2) + \Delta_2$.

This implies that the relation 
\[
\Delta_1' = \sum_{i=2}^{h_1(\Gamma)} \Delta_1' \cap \Delta_i'= \Delta_1' \cap \Delta_2 + \sum_{i=3}^{h_1(\Gamma)} \Delta_1' \cap \Delta_j
\]
from the beginning yields the equation
\[
\Delta_1=\Delta_1 \cap \Delta_2 + \sum_{j \ge 3} (\Delta_1+\Delta_2) \cap \Delta_j. 
\]
This is a non-trivial relation among the elements of 
\[
\{\Delta_1,\, \ldots,\, \Delta_{h_1(\Gamma)},\, \Delta_i \cap \Delta_j \colon \forall i,\, j\},
\]
i.e., matrix entries of $M$, a contradiction. 

Hence, $\Gamma$ is planar, and therefore polygonal by Lemma~\ref{poly_versus_planar}.  
Assume $\Gamma= \Delta_1 \cup \Delta_2 \cup \cdots \cup \Delta_h$, but $h_1(E_0)>0$. 
Let $\Delta_1,...,\Delta_h$ be the natural basis of $H_1(\Gamma)$ given by the cycles $\Delta_i$.  
Given a simple non-zero loop $\gamma \subset E_0$, there is a linear relation between the diagonal entries for all 
$\Delta_i$ meeting $\gamma$ and all off-diagonal entries carrying glueing data for these $\Delta_i$. 

(i) $\Rightarrow$ (ii): Conversely, suppose that $\Gamma$ is a $*$-graph and we have given a linear relation among the entries of $M_{\Gamma,B}$. 
By definition of $*$-graphs, this relation 
involves a diagonal element, since every edge is only used once for glueing. Hence, we get an equation
\[
\sum_{i=1}^h a_i M_{ii}= \sum_{i<j} b_{ij} M_{ij}, 
\]
with at least one $a_i$ and one $b_{ij}$ non-zero by Lemma \ref{MatrixStructure}. This is a contradiction, since 
each $\Delta_i$ occurring on the left with $a_i \neq 0$ has an edge which is not contained in $E_0$.
\end{proof}

\begin{corollary} The $\ast$-graphs admit a torus operation of dimension $r \ge h-1+n-\ell(M_{\Gamma,B})$. It is faithful under the condition given in 
Theorem~\ref{Existence}, i.e., if the graph hypersurface is not a union of $h$ linear hyperplanes. 
\end{corollary}

\begin{proof} By Prop.~\ref{polys}, the entries of $M_{\Gamma,B}$ satisfy the assumptions of Thm.~\ref{Existence}.  
\end{proof}

\section{Motivic Bialynicki-Birula decompositions}

In this section we discuss how to apply high dimensional torus operations on $X_\Gamma$ to compute the motive of a 
graph hypersurface $X_\Gamma=\{\det(M_\Gamma)=0\}$ using a motivic version of the decomposition theorem of Bialynicki-Birula \cite{BB}.
For simplicity assume that $k$ is algebraically closed and of characteristic zero. 

In the following we use (cohomological) motives $M(X)$ in the sense of Voevodsky's triangulated category $DM(k)=DM_{gm}(k)$ attached to any $k$-scheme $X$. 
The motive $M(X)$ for a possibly singular variety $X$ is defined in \cite[chap. 5]{V}.
We want to give a criterion when the motive of a graph hypersurface $M(X_\Gamma) \in DM(k)$ is mixed Tate. 
An object $M \in DM(k)$ is called mixed Tate, if it is in the image of 
\[
DMT(k) \rightarrow DM(k) \otimes \Q,
\]
where $DMT(k)$ is the full $\Q$-linear triangulated subcategory of $DM(k) \otimes \Q$ generated by the Tate objects $\Q(n)$ as 
defined by Levine \cite{L}. 

\begin{example}
The simplest example which is not entirely trivial is $\Gamma = WS_3$, the wheel with $3$ spokes.
The graph hypersurface $X_\Gamma$ for $\Gamma= WS_3$ is isomorphic to ${\rm Sym}^2 \P^2$, and admits a $2$-dimensional torus operation. 
The motive of $X_\Gamma$ is mixed Tate by \cite[Sect. 9]{B}.
\end{example}

In view of the classical Bialynicki-Birula theorem \cite{BB} and its motivic versions \cite{Br,K},
one might expect that the motive of $X_\Gamma$ should be determined by the components $F$ of the fixed point set, 
if $X_\Gamma$ carries a non-trivial torus operation. In the smooth case, the theorem of Bialynicki-Birula takes the form
\[
M(X) \cong \bigoplus_F M(F)(n_F)
\]
in $DM(k)$ with appropriate Tate twists $n_F$ depending on each $F$. 
In the presence of singularities, we have to use equivariant cubical hyperresolutions to obtain a useful version ob Bialynicki-Birula's theorem.
The idea is to replace a singular variety $X$ by a simplicial variety $X_\bullet \to X$ with smooth components $X_\alpha$ and 
an equivariant torus operation on each $X_\alpha$. 

\begin{proposition} 
For every integral closed subvariety $X \subset \P^{n-1}$ with an algebraic operation of a torus $T$, there is an equivariant cubical hyperresolution 
\[
X_\bullet \longrightarrow X 
\]
in the sense of \cite{GN}. Every component $X_\alpha$ in the hyperresolution $X_\bullet$ can be chosen 
smooth and projective. The motive $M(X) \in DM(k)$ can be obtained from $X_\bullet$ by descent, i.e., 
the morphism $M(X) \to M(X_\bullet)$ is an isomorphism.
\end{proposition}

\begin{proof} See \cite{GN} or \cite[Thm. 5.2.6]{PS} for the explicit construction of a cubical hyperresolution via a resolution of singularities. 
The construction is inductive, and in each step some varieties are replaced by several smooth components.  
Levine has used this in the context of motives, see for example \cite[Thm.~3.2.5, pg. 246]{L2}. For $X$ one has now two motives: $M(X)$ as defined in \cite[chap. 5]{V}, 
and $M(X):=M(X_\bullet)$ as defined in Levine. However, there is a descent statement for the cdh-topology in $DM(k)$ \cite[chap.5, sect.4]{V}, 
and this implies, by inductive application in the abtract blow-up squares of a cubical hyperresolution, 
that $M(X)$ and $M(X_\bullet)$ are isomorphic in the triangulated category $DM(k)$. 
A resolution of singularities, hence a cubical hyperresolution, can be made equivariant using equivariant resolution of singularities, see \cite{W}.
\end{proof}

\begin{example} A nice example with a torus operation 
is the nodal rational curve $C$ with desingularization $\P^1$, where 
the points $0$ and $\infty$ on $\P^1$ are identified to the singular point $P$ in $C$. 
The associated cube is the square
\[
\begin{matrix} \{0,\infty\} & \longrightarrow & \{P\}  \cr 
\downarrow & & \downarrow \cr
\P^1 & \longrightarrow & C.
\end{matrix}
\]  
Over a perfect field of positive characteristic, alterations in the sense of de Jong 
give another way of constructing such a hyperresolution.    
\end{example}

We assume that we are in this situation now.

\begin{proposition} \label{DM} Assume that all fixed point loci in all smooth, proper components $X_\alpha$ of 
$X_\bullet$ induce mixed Tate motives $M(X_\alpha)$. Then $M(X)$ is mixed Tate.
\end{proposition}

\begin{proof} All components $X_\alpha$ in the cubical hyperresolution give a mixed Tate motive $M(X_\alpha)$ by assumption. The arrows in 
the simplicial variety $X_\bullet$ are contained in the full subcategory $DMT(k)$. Hence 
$M(X_\bullet)$ descends to a mixed Tate motive $M(X)$. 
\end{proof}

Proposition \ref{DM} reduces the complexity of the motive of $X_\Gamma$ with this method to that of the fixed point 
loci in some resolution of singularities. This method should 
be successful provided there is some sufficiently high dimensional torus operation. 

\begin{example} Let us revisit $\Gamma = WS_3$, the wheel with $3$ spokes.
The graph hypersurface $X_\Gamma$ for $\Gamma= WS_3$ is isomorphic to ${\rm Sym}^2 \P^2$, which has a 
resolution by a single blow-up of the diagonal. By Lemma~\ref{Fixpoints}, the fixed point locus ${\rm Fix}_{\P^{5}}(T)$ consists of points, hence
$M(X_\Gamma)$ is a mixed Tate motive. 
\end{example}

However, besides the wheel with $3$ spokes, we do not have many examples.
Note that the equivariant resolution of a singular hypersurface $X$ can have a fixed point set which is a not mixed Tate motive, even if
the fixed point set in $X$ consists of isolated points. The cone over an elliptic curve gives such an example. We ask 
the following question:

\begin{question} Assume that $X=X_\Gamma$ is a graph hypersurface with algebraic torus operation. 
Is there always an equivariant cubical hyperresolution $X_\bullet$, 
such that each smooth stratum $X_\alpha$ is defined in graph theoretic terms, and the fixed point loci for the torus operation on $X_\alpha$ can be computed 
in terms of graph invariants ? 
\end{question}

{\bf Acknowledgment:} This work was supported by Sonderforschungsbereich TRR 45 of Deutsche Forschungsgemeinschaft.  
We thank S. Bloch, M. Brion, P. Brosnan, H. Esnault, M. Lehn, M. Levine, V. Welker and J. Winkelmann for discussions, and 
the referees for many helpful suggestions, and improvements.

\end{document}